%% file: SteinhausV4_12.tex
\documentclass[a4paper,12pt]{amsart}

\usepackage{graphicx}
\usepackage{xcolor}
\usepackage{amssymb,amsmath,amsfonts}
\usepackage[active]{srcltx}
\usepackage{hyperref}
\usepackage[T1]{fontenc}

\usepackage{color}
\definecolor{red}{rgb}{1,0,0}

\definecolor{gre}{rgb}{0,0.7,0}

\newtheorem{thm}{Theorem}%
\newtheorem{lem}{Lemma}
\newtheorem{prop}{Proposition}%
\newtheorem{cor}{Corollary}%
\theoremstyle{definition}

\theoremstyle{remark}
\theoremstyle{plain}

\newcommand{\R}{\mathbb R}
\newcommand{\Z}{\mathbb Z}
\newcommand{\Q}{\mathbb Q}
\newcommand{\rar}{\rightarrow}

\def\NN{{\mathbb N}}
\def\QQ{{\mathbb Q}}

\def\RR{{\mathbb R}}
\def\TT{{\mathbb T}}

\def\ZZ{{\mathbb Z}}
\def\Z{{\mathbb Z}}

\def\scrA{{\mathcal A}}

\def\scrC{{\mathcal C}}
\def\scrD{{\mathcal D}}

\def\scrG{{\mathcal G}}
\def\scrI{{\mathcal I}}

\def\scrJ{{\mathcal J}}
\def\scrK{{\mathcal K}}

\def\scrM{{\mathcal M}}

\def\scrQ{{\mathcal Q}}

\def\scrU{{\mathcal U}}
\def\scrV{{\mathcal V}}

\def\e{\mathrm{e}}
\def\i{\mathrm{i}}

\def\diag{\operatorname{diag}}

\def\cl{\operatorname{cl}}

\def\GL{\operatorname{GL}}

\def\SL{\operatorname{SL}}

\def\GamG{\Gamma\backslash G}

\def\trans{\,^\mathrm{t}\!}

\numberwithin{equation}{section}

\begin{document}

\title[Higher dimensional Steinhaus and Slater problems]{Higher dimensional Steinhaus and Slater\\ problems via homogeneous dynamics}
\author{Alan Haynes, Jens Marklof}
\date{13 July 2017; revised 17 April 2018}

\thanks{AH: Research supported by EPSRC grants EP/L001462, EP/M023540.\\
\phantom{A..}JM: The research leading to these results has received funding from the
European Research Council under the European Union's Seventh Framework
Programme (FP/2007-2013) / ERC Grant Agreement n. 291147.\\
\phantom{A..}MSC 2010: 11J13, 60D05}

\begin{abstract}
The three gap theorem, also known as the Steinhaus conjecture or three distance theorem, states that the gaps in the fractional parts of $\alpha,2\alpha,\ldots, N\alpha$ take at most three distinct values. Motivated by a question of Erd\H{o}s, Geelen and Simpson, we explore a higher-dimensional variant, which asks for the number of gaps between the fractional parts of a linear form. Using the ergodic properties of the diagonal action on the space of lattices, we prove that for almost all parameter values the number of distinct gaps in the higher dimensional problem is unbounded. Our results in particular improve earlier work by Boshernitzan, Dyson and Bleher et al. We furthermore discuss a close link with the Littlewood conjecture in multiplicative Diophantine approximation. Finally, we also demonstrate how our methods can be adapted to obtain similar results for gaps between return times of translations to shrinking regions on higher dimensional tori.
\end{abstract}

\maketitle

\section{Introduction}\label{sec:one}

\subsection{The Steinhaus problem}

Let $\scrD\subset\RR^d$ be a bounded convex set. For $\alpha\in\RR^d$, define
\begin{equation} S(\alpha,\scrD)=\{m\cdot\alpha\bmod 1 \mid  m\in \ZZ^d \cap \scrD \} \subset \RR/\ZZ, \end{equation}
and let $G(\alpha,\scrD)$ be number of distinct gaps between the elements of $S(\alpha,\scrD)$. In other words, the set $S(\alpha,\scrD)$ partitions $\R/\Z$ into intervals of $G(\alpha,\scrD)$ distinct lengths.

In the classical case $d=1$, the three gap theorem (also referred to as {\em Steinhaus conjecture} or {\em three distance theorem}) asserts that for all $\alpha\in\RR$ and any interval $\scrD$, we have
$G(\alpha,\scrD) \leq 3$.
The first proofs of this remarkable fact were published in 1957 by S\'{o}s \cite{Sos1957}, in 1958 by Sur\'{a}nyi \cite{Sura1958}, and in 1959 by \'{S}wierczkowski \cite{Swie1959}.
The theorem has been rediscovered repeatedly, and many authors have considered generalizations to various settings \cite{BaloGranSoly2017,Chev2014,CobeGrozVajaZaha2002,FrieSos1992,HaynKoivSaduWalt2016,HaynKoivWalt2016,Lang1991,MarkStro2017,Rave1988,Slat1967,Vija2008}.

In this paper we are firstly interested in a higher dimensional version of the Steinhaus problem, which was previously studied by Geelen and Simpson \cite{GeelSimp1993}, Fraenkel and Holzman \cite{FraeHolz1995}, Chevallier \cite{Chev2000}, Boshernitzan \cite{Bosh1991,Bosh1992}, Dyson \cite{Dyson}, and Bleher, Homma, Ji, Roeder, and Shen \cite{BlehHommJiRoedShen2012}. For this problem our goal is twofold: to demonstrate the close connection between the multi-dimensional Steinhaus problem and the Littlewood conjecture, and to show how well known results from ergodic theory on the space of unimodular lattices in $\R^d$ can be used to shed new light on a question of Erd\H{o}s as stated by Geelen and Simpson \cite[Section 4]{GeelSimp1993}.

Our first theorem describes the generic failure of the finite gap phenomenon in higher dimensions.
Denote by $R\scrD=\{ R x \mid x\in\scrD \}$ the homothetic dilation of $\scrD$ by a factor of $R$.
We say a sequence $0<R_1< R_2 < R_3 <\ldots$ is {\em subexponential} if
\begin{equation}\lim_{i\to\infty} R_i = \infty,\qquad \lim_{i\to\infty} \frac{R_{i+1}}{R_i}=1 .\end{equation}

\begin{thm}\label{thm:one}
Let $d\geq 2$.
There exists a set $P\subset\RR^d$ of full Lebesgue measure, such that for every bounded convex $\scrD\subset\RR^d$ with non-empty interior, every $\alpha\in P$, and every subexponential sequence $(R_i)_i$, we have
\begin{equation} \label{diverge}
\sup_i G(\alpha,R_i\scrD)=\infty
\end{equation}
and
\begin{equation} \label{bdd}
\liminf_i G(\alpha,R_i\scrD)<\infty .
\end{equation}
\end{thm}

A previous result in this direction is due to Bleher, Homma, Ji, Roeder, and Shen  \cite{BlehHommJiRoedShen2012}, who show in the case $d=2$, and for a certain set of $\alpha$, that 
\begin{equation}\label{eqn.InfNNGaps}
\sup_{R\geq 1} G(\alpha,R\scrD)=\infty,
\end{equation}
where $\scrD$ is the triangle in $\R^2$ with vertices at $(0,0),(0,1),$ and $(1/2,0)$. For purposes of comparison with Theorem \ref{thm:one}, a careful computation shows that the size of the set of $\alpha$ to which the proof in \cite{BlehHommJiRoedShen2012} applies, has Hausdorff dimension $3/2$. (For the details of this computation, the reader may consult Lemma 6.1 of \cite{HaynKoivSaduWalt2016} and the paragraphs immediately following its proof.) Theorem \ref{thm:one} on the other hand admits a set of $\alpha$ of full Hausdorff dimension $d$.

In the case $d=2$, for $\scrD =[0,1)^2$ a square, a folklore problem of Erd\H{o}s (see the discussion at the end of \cite{GeelSimp1993}) asks whether eq.~\eqref{eqn.InfNNGaps} holds
whenever $1,\alpha_1,\alpha_2$ are $\Q$-linearly independent.
The answer to this question is in fact, negative. As recorded in \cite{BlehHommJiRoedShen2012}, this appears to have first been noticed in a private correspondence between Freeman Dyson and Michael Boshernitzan \cite{Bosh1991,Bosh1992,Dyson}, who showed that \eqref{eqn.InfNNGaps} fails for badly approximable $\alpha$.

We say that $\alpha\in\RR^d$ is {\em badly approximable} if there is $c>0$ such that $\| m\cdot\alpha \|_{\RR/\ZZ} > c \| m \|^{-d}$ for all non-zero $m\in\ZZ^d$. Here $\| x \|_{\RR/\ZZ}=\min_{k\in\ZZ} \| x+k\|$ denotes the distance to the nearest integer.

\begin{thm}[Boshernitzan and Dyson; Bleher, Homma, Ji, Roeder, and Shen]\label{thm:two}
Let $d\geq 2$.
For every bounded convex $\scrD\subset\RR^d$ with non-empty interior, and every badly approximable $\alpha\in\RR^d$, we have \begin{equation}\sup_{R\geq 1} G(\alpha,R\scrD)<\infty.\end{equation}
\end{thm}

We will see below that this statement is an immediate consequence of our dynamical interpretation of $G(\alpha,R\scrD)$ combined with Dani's correspondence between badly approximable numbers and bounded orbits in the space of lattices.

Let us now turn to the connection between the Steinhaus problem and the Littlewood conjecture in multiplicative Diophantine approximation. The Littlewood conjecture states that for {\em every} $\alpha_1,\alpha_2\in\R$,
\begin{equation}\label{eq:little00}
\liminf_{n\to\infty} n \| n \alpha_1 \|_{\RR/\ZZ} \| n \alpha_2\|_{\RR/\ZZ} = 0.
\end{equation}
There is a higher dimensional version of this conjecture, that for any $d\geq 2$ and for {\em every} $\alpha\in\RR^d$,
\begin{equation}\label{eq:little0}
\liminf_{n\to\infty} n \| n \alpha_1 \|_{\RR/\ZZ} \cdots \| n \alpha_d\|_{\RR/\ZZ} = 0.
\end{equation}
Resolving the conjecture for $d=2$ would imply the higher dimensional statement for all $d>2$, but at present the conjecture has not been proved in full for any value of $d$. However, it is known that \eqref{eq:little0} holds for a set of $\alpha$ whose complement has Hausdorff dimension zero \cite{EinsKatoLind2006}.

Consider the (in general non-homogeneous) dilation $\scrD_T=\{ x T \mid x\in\scrD \}$ of $\scrD$, where $T=\diag(T_1,\ldots,T_d)$ is a diagonal matrix with expansion factors $T_i>0$.

\begin{thm}\label{thm:three}
Let $d\geq 2$. Assume $\scrD\subset\RR^d$ is bounded convex and contains the cube $[0,\epsilon)^d$ for some $\epsilon>0$. If $\alpha=(\alpha_1,\ldots,\alpha_d)\in\RR^d$ is such that
\begin{equation}\label{eq:little2}
\sup_{T_1,\ldots,T_d\geq 1} G(\alpha,\scrD_{T})=\infty,\end{equation}
then
\begin{equation}\label{eq:little}
\liminf_{n\to\infty} n \| n \alpha_1 \|_{\RR/\ZZ} \cdots \| n \alpha_d\|_{\RR/\ZZ} = 0.
\end{equation}
\end{thm}

Theorem \ref{thm:one} implies that eq.~\eqref{eq:little2} holds for a set of $\alpha$ of full Lebesgue measure. We expect that there is a more concise characterisation of the set of exceptions, in analogy to the case of the Littlewood conjecture. But, unlike the Littlewood conjecture, eq.~\eqref{eq:little2} is not true for all $\alpha$. This is obvious for $\alpha\in\QQ^d$. The following theorem gives a less trivial class of examples.

\begin{thm}\label{thm:five}
Suppose $\alpha= r \beta + s$, with $r,s\in\QQ^d$, $\beta\in\RR$, and let $\scrD=[0,1)^d$ be the unit cube in $\R^d$. Then we have that \begin{equation}\label{eqn.SupFinite}\sup_{T_1,\ldots,T_d\geq 1} G(\alpha,\scrD_T)<\infty.\end{equation}
\end{thm}

It follows from \cite[Theorem 2]{Burg2000} that, for $d\ge 2$, if  $\alpha$ satisfies the hypotheses of Theorem \ref{thm:five} then eq.~\eqref{eq:little} (and in fact a much stronger statement) holds for $\alpha$. Therefore this theorem highlights the difference between the exceptional sets for the Steinhaus problem and for the Littlewood conjecture. This will be reflected in our dynamical interpretation: while there is a one-to-one correspondence of $\alpha$ satisfying the Littlewood conjecture and unbounded orbits in the space of lattices \cite[Prop.~11.1]{EinsKatoLind2006}, infinite gaps in the Steinhaus problem require in addition a particular type of divergence in the space of lattices.

\subsection{The Slater problem}

In addition to our results for the higher dimensional Steinhaus problems, our methods allow us to easily deduce results for a dual collection of problems, which we now describe. For $\alpha\in\RR^d$, consider the toral translation
\begin{equation}
\TT^d \mapsto \TT^d, \qquad q \mapsto q+\alpha ,
\end{equation}
where $\TT^d=\RR^d/\ZZ^d$.
Let $\scrD\subset\RR^d$ be a convex open set which is contained in a bounded fundamental domain of $\R^d/\ZZ^d$; e.g. $\scrD\subset (-\tfrac12,\tfrac12]^d$. For $q\in\scrD$, the first return time to $\scrD$ is given by
\begin{equation}
\tau(q,\scrD) = \min\{ n\in \NN^* \mid q+ n\alpha \in\scrD + \ZZ^d \},
\end{equation}
where $\NN^*$ denotes the natural numbers without 0.
We are interested in the number $L(\alpha,\scrD)$ of distinct values $\tau(q,\scrD)$ attains, as $q$ varies over $\scrD$, and whether that number remains finite as $\scrD$ contracts. The problem in dimension $d=1$ was studied by Slater in 1950 \cite{Slat1950,Slat1967} and is closely related to the three gap theorem. Indeed the answer is $L(\alpha,\scrD)\leq 3$, for any $\alpha$ and interval $\scrD$. This fact was later rediscovered in the study of Thom's problem for the linear flow on a flat two-dimensional torus \cite{BlanKrik1993}, and a number of generalizations and extensions of the theorem are discussed in \cite{FraeHolz1995}. Following \cite{FraeHolz1995}, we refer to these types of problems as Slater problems. The analogues of Theorems \ref{thm:one}--\ref{thm:three} for the higher dimensional Slater problems are as follows.

\begin{thm}\label{thm:oneB}
Let $d\geq 2$, take $P$ to be the set from Theorem \ref{thm:one}, and assume (as we may, without loss of generality) that $P=-P$. Then, for every bounded convex $\scrD\subset\RR^d$ with non-empty interior, for every $\alpha\in P$, and every subexponential sequence $(R_i)_i$, we have
\begin{equation}
\sup_i L(\alpha,R_i^{-1}\scrD)=\infty
\end{equation}
and
\begin{equation}
\liminf_i L(\alpha,R_i^{-1}\scrD)<\infty .
\end{equation}
\end{thm}

\begin{thm}\label{thm:twoB}
Let $d\geq 2$.
For every bounded convex $\scrD\subset\RR^d$ with non-empty interior, and every badly approximable $\alpha\in\RR^d$, we have \begin{equation}\sup_{R\geq 1} L(\alpha,R^{-1}\scrD)<\infty.\end{equation}
\end{thm}

\begin{thm}\label{thm:threeB}
Let $d\geq 2$ and $\scrD\subset\RR^d$ be bounded and convex with non-empty interior. If $\alpha=(\alpha_1,\ldots,\alpha_d)\in\RR^d$ such that
\begin{equation}\label{eq:little2B}
\sup_{T_1,\ldots,T_d\geq 1} L(\alpha,\scrD_{T^{-1}})=\infty,\end{equation}
then
\begin{equation}\label{eq:littleB}
\liminf_{n\to\infty} n \| n \alpha_1 \|_{\RR/\ZZ} \cdots \| n \alpha_d\|_{\RR/\ZZ} = 0.
\end{equation}
\end{thm}

\subsection{Outline}

The plan of this paper is as follows. Motivated by the approach of \cite{MarkStro2017} in the case $d=1$, we first provide an interpretation of $G(\alpha,\scrD_{T})$ as a certain function on the space of $(d+1)$-dimensional Euclidean lattices (Section \ref{sec:two}). This will allow us to derive Theorems \ref{thm:one}--\ref{thm:three} from dynamical properties of the diagonal action on the space of lattices (Section \ref{sec:three}). The proof of Theorem \ref{thm:five}, which is presented in Section \ref{sec:four}, involves a reduction to a theorem of Chevallier \cite[Theorem 1]{Chev2000}. Finally, the proofs of Theorems \ref{thm:oneB}-\ref{thm:threeB} will be given in Section \ref{sec:Slater}.

\subsection{Acknowledgements}
We would like to thank Nicolas Chevallier and an anonymous referee for a number of detailed comments which helped us to improve upon a preliminary version of this paper.

\section{The Steinhaus problem in terms of the space of lattices} \label{sec:two}

Given $\alpha\in\RR^d$ and $k\in\Z^d$, denote by $\xi_k=k\cdot\alpha \bmod 1$ the fractional part of $k\cdot\alpha$.
Assume in the following that $\scrD\subset\RR^d$ is bounded and has non-empty interior.
Set $\scrD_B=\{ x B \mid x\in\scrD\}$ with $B\in\GL(d,\RR)$, $\det B>0$. We now follow the strategy developed in \cite{MarkStro2017} for the case $d=1$.

For $k\in\scrD_B\cap\ZZ^d$, the gap between $\xi_k$ and its next neighbor on $\RR/\ZZ$ is given by
\begin{equation}\label{g2nn}
s_{k,B}  = \min\{ (\ell-k)\cdot\alpha + n > 0 \mid (\ell,n)\in\ZZ^{d+1},\; \ell\in \scrD_B \} .
\end{equation}
The substitution $m=\ell-k$ yields
\begin{equation}\label{0st}
s_{k,B}  = \min\{ m \cdot\alpha + n > 0 \mid (m,n)\in\ZZ^{d+1},\; m+k\in \scrD_B \} ,
\end{equation}
which we rewrite as
\begin{equation}\label{resca}
s_{k,B} = \min\{ y > 0 \mid (x,y)\in\ZZ^{d+1} A_1 ,\; x+k\in \scrD_B  \},
\end{equation}
with the matrix
\begin{equation}\label{lattice0}
A_1=\begin{pmatrix} 1_d & \trans\alpha \\ 0 & 1 \end{pmatrix} .
\end{equation}

Let $G=\SL(d+1,\RR)$ and $\Gamma=\SL(d+1,\ZZ)$. Now take a general element $M\in G$ and $t\in \scrD$, and define the function $F$ by
\begin{equation}\label{Fdef}
F(M,t)=\min\big\{ y > 0 \;\big|\; (x,y)\in\ZZ^{d+1} M, \; x+t\in\scrD \big\} ,
\end{equation}
whenever the minimum exists, and by $F(M,t)=\infty$ otherwise.
(Proposition \ref{prop1} below establishes that the minimum exists for all $t\in\scrD^\circ$.)
To see the connection of $F$ with the gap $s_{k,B}$, define
\begin{equation}\label{lattice}
A_B=\begin{pmatrix} 1_d & \trans\alpha \\ 0 & 1 \end{pmatrix} \begin{pmatrix} B^{-1} & 0 \\ 0 & \det B \end{pmatrix} \in G,
\end{equation}
and note that, by rescaling the set in \eqref{resca}, we have
\begin{equation}
s_{k,B} = (\det B)^{-1} \min\big\{ y > 0 \;\big|\; (x,y)\in\ZZ^{d+1} A_B , \; x+k B^{-1}\in\scrD \big\}.
\end{equation}
Thus, 
\begin{equation}\label{key}
s_{k,B}=(\det B)^{-1}  F\big(A_B,k B^{-1}\big).
\end{equation}

\begin{prop}\label{prop1}
$F$ is well-defined as a function $\GamG \times \scrD^\circ \to \RR_{> 0}$.
\end{prop}

\begin{proof}
Let us begin by showing that
\begin{equation}\label{theset}
\big\{ y > 0 \;\big|\; (x,y)\in\ZZ^{d+1} M,\; x+t\in\scrD  \big\}
\end{equation}
is non-empty for every $M\in G$, $t\in\scrD^\circ$. Since $\scrD^\circ$ is open, for every given $t\in\scrD^\circ$ there is $\epsilon>0$ such that $x+t\in\scrD$ for all $\|x\|<\epsilon$. There are at most finitely many lattice points $(x,y)\in\ZZ^{d+1} M$ with $\|x\|<\epsilon$ and $y=0$. By decreasing $\epsilon$ further, we can ensure that $0$ is the only such point. It follows from Minkowski's theorem that the infinite cylinder $\{ (x,y)\in\RR^d\times\RR: \|x\|<\epsilon\}$ contains a non-zero lattice point in $\ZZ^{d+1}M$. Therefore, since the lattice is symmetric with respect to reflection at the origin, also the semi-infinite $\{ (x,y)\in\RR^d\times\RR_{\geq 0}: \|x\|<\epsilon\}$ contains a non-zero lattice point $(x,y)$ for every $\epsilon>0$. By construction, $y\neq 0$. This implies \eqref{theset} is non-empty. The minimum exists in view of the uniform discreteness of $\ZZ^{d+1} M$.

Finally, we note that $F(\,\cdot\,,t)$ is well-defined as a function on $\GamG$ since $F(M,t)=F(\gamma M,t)$ for all $M\in G$, $\gamma\in\Gamma$.
\end{proof}

Denote by $\Delta\scrD$ the set of differences $\{s-t\mid s,t \in\scrD \}$, and set
\begin{equation}\label{theset234}
\scrM(M) =
\big\{ y > 0 \;\big|\; (x,y)\in\ZZ^{d+1} M,\; x\in\Delta\scrD  \big\} ,
\end{equation}
which contains the set of values of $F(M,\,\cdot\,)$.
Since $\Delta\scrD$ is bounded and $\ZZ^{d+1} M$ is uniformly discrete, $\scrM(M)$ is a locally finite  subset of $\RR_{> 0}$.

In particular, for every fixed $M$, the function $t\mapsto F(M,t)$ is therefore piecewise constant. We furthermore have the following.
\begin{prop}\label{prop:cont}
Let $\scrC\subset\GamG\times\scrD^\circ$ be compact. Then (i) there exists a positive $\kappa(\scrC)$ such that $F(M,t) < \kappa(\scrC)$ if $(\Gamma M,t)\in \scrC$, and (ii) $F$ is continuous at $(\Gamma M,t) \in \scrC$ if
\begin{equation}\label{exset}
(\ZZ^{d+1} M\setminus\{0\})
\cap \partial((\scrD - t)\times [0,\kappa(\scrC)]) = \emptyset .
\end{equation}
\end{prop}

\begin{proof}
(i) We will use a quantitative variant of the proof of Proposition
\ref{prop1}. 
Since $\scrC$ is compact, there is an $\epsilon>0$ such that $x+t\in\scrD$ for all $\|x\|<\epsilon$
and $(\Gamma M,t)\in\scrC$. Furthermore, by Mahler's compactness criterion
there exists $\epsilon_0>0$ such that $\| (x,y) \| >\epsilon_0$ for
every $(x,y)\in\ZZ^{d+1} M\setminus\{0\}$, uniformly over
$(\Gamma M,t)\in\scrC$. By Minkowski's theorem, the cylinder
\begin{equation}
\{ (x,y)\in\RR^{d+1} \mid \|x\|<\min(\epsilon_0,\epsilon),\; 0<y< \kappa(\scrC) \}
\end{equation}
contains at least one lattice point in $\ZZ^{d+1} M\setminus\{0\}$,
where $\kappa(\scrC)$ is any constant greater than $2^d
V_d^{-1}\min(\epsilon_0,\epsilon)^{-d}$, and $V_d$ is the volume of
the unit ball in $\RR^d$. Hence $F(M,t)< \kappa(\scrC)$ if $(\Gamma M,t)\in\scrC$,
as required.

(ii) 
By Mahler's criterion, all points in
$\ZZ^{d+1} M$ are at least distance $\epsilon_0$ apart (with the same $\epsilon_0>0$ as in part (i)), for all
$(\Gamma M,t)\in\scrC$.  Define the compact set
$\scrK:=\cl(\Delta\scrD)\times [0,\kappa(\scrC)]$ with $\kappa(\scrC)$ as in part
(i). Suppose $(M_i,t_i)\to (M,t)$ for some sequence of
$(M_i,t_i)\in G\times\scrD^\circ$. Because $\scrK$ is compact, $(\ZZ^{d+1}
M\setminus\{0\}) \cap \scrK$ is finite;  hence $m M_i\to m M$ uniformly for all points $m M\in (\ZZ^{d+1}
M\setminus\{0\}) \cap \scrK$.  Therefore, given $\epsilon\in(0,\epsilon_0)$, there is $i_0$
such that for every $i\geq i_0$, we have that every open
$\epsilon$-ball centered at a lattice point in $(\ZZ^{d+1}
M\setminus\{0\}) \cap \scrK$ contains precisely one point in
$\ZZ^{d+1} M_i\setminus\{0\}$. By assumption \eqref{exset}, every lattice point in $(\ZZ^{d+1}
M\setminus\{0\}) \cap \scrK$ which lies in $(\scrD - t) \times\RR_{\geq 0}$, is contained in the open set
$(\scrD^\circ - t) \times\RR_{> 0}$. Furthermore there is
$\epsilon_1>0$ such that, for every lattice point in $(\ZZ^{d+1}
M\setminus\{0\}) \cap \scrK$ contained in the open set
$(\scrD^\circ - t) \times\RR_{> 0}$,
the open $\epsilon$-ball centered at this lattice point is also
contained in $(\scrD^\circ - t) \times\RR_{> 0}$ for every
$\epsilon<\epsilon_1$, and hence the open $\frac{\epsilon}{2}$-ball
centered at that lattice point is contained in $(\scrD^\circ - t_i)
\times\RR_{> 0}$ provided $i$ is sufficiently large so that
$|t-t_i|<\frac{\epsilon}{2}$. Thus, given
$0<\epsilon<\min(\epsilon_0,\epsilon_1)$, we have
$|F(M_i,t_i)-F(M,t)|<\epsilon$ for all sufficiently large $i$.
\end{proof}

Given a bounded subset $\scrA\subset\RR^{d+1}$ with non-empty interior, and $M\in G$, we define the {\em covering radius} (also called {\em inhomogeneous minimum})
\begin{equation}
\rho(M,\scrA)=\inf\{ \theta>0 \mid \theta\scrA+\ZZ^{d+1} M = \RR^{d+1} \}.
\end{equation}
Because $\scrA$ has non-empty interior, $\rho(M,\scrA)<\infty$. Assume now that $\scrA$ is convex. Then we have $\theta\scrA+\ZZ^{d+1} M = \RR^{d+1}$ for every $\theta>\rho(M,\scrA)$, and hence the set $\theta\scrA + x$ intersects $\ZZ^{d+1} M$ in at least one point, for every $x\in\RR^{d+1}$. (To see this, assume the contrary: There is $x\in\RR^{d+1}$ such that $(\theta\scrA+x) \cap \ZZ^{d+1} M = \emptyset$. So $(\theta\scrA+\ZZ^{d+1} M) \cap (\ZZ^{d+1} M-x) = \emptyset$, contradicting our assumption that $\theta\scrA+\ZZ^{d+1} M = \RR^{d+1}$.)  For a given set $\scrC\subset\GamG$, we define
\begin{equation}
\overline\rho(\scrC,\scrA) = \sup_{\Gamma M\in\scrC} \rho(M,\scrA).
\end{equation}
It is well known that $\overline\rho(\scrC,\scrA)<\infty$ for every compact $\scrC\subset\GamG$.
For $\theta>0$, set
\begin{equation}\label{D-def}
D(\theta) =
\begin{pmatrix}
\theta 1_d & \trans 0 \\ 0 & \theta^{-d}
\end{pmatrix}
\in G.
\end{equation}


\begin{prop}\label{prop:twoB}
Let $\scrD$ be bounded and convex with non-empty interior. Assume $\scrC\subset\GamG$ is compact, and $\theta>\overline\rho(\scrC,\scrD\times(0,1])$. Then
\begin{equation}\label{eqn.FUppBd}
F(M,t) \leq \theta^{d+1}
\end{equation}
for $\Gamma M\in \scrC D(\theta)^{-1}$ and $t\in\scrD$.
\end{prop}

\begin{proof}
Set $\scrA_{t,\theta}=(\scrD-t)\times (0,\theta^{d+1}]$. The task is to show that $\scrA_{t,\theta}$ intersects $\ZZ^{d+1} M$ in at least one point, for every $\Gamma M\in\scrC D(\theta)^{-1}$ and every $t\in\scrD$. Now $\scrA_{t,\theta}\cap\ZZ^{d+1} M\neq \emptyset$ is equivalent to $\theta\scrA_{t,1}\cap\ZZ^{d+1} M D(\theta) \neq \emptyset$. The latter holds because the assumption that $\theta>\overline\rho(\scrC,\scrD\times (0,1])$ implies that $\theta\scrA_{t,1}\cap\ZZ^{d+1} M' \neq \emptyset$ for every $\Gamma M'\in\scrC$, and $\Gamma M'=\Gamma M D(\theta) \in\scrC$ by assumption.
\end{proof}

We denote by $\scrG(M)$ the number of distinct values the function $t\mapsto F(M,t)$ attains, as $t$ runs over $\scrD$. For $R>0$, let $\scrG_R(M)$ be the number of distinct values of $F(M,R^{-1} k)$ as $k\in\ZZ^d$ runs over $R\scrD$. We have of course $\scrG_R(M)\leq \scrG(M)$.

\begin{prop}\label{prop:three}
Let $\scrD$ be bounded and convex with non-empty interior. Assume $\scrC\subset\GamG$ is compact, and $\theta>\overline\rho(\scrC,\scrD\times(0,1])$. Then there is a constant $C_\theta<\infty$ such that
\begin{equation}
\scrG(M) \leq C_\theta
\end{equation}
for $\Gamma M\in\scrC D(\theta)^{-1}$.
\end{prop}

\begin{proof}
Note that $\scrG(M)$ is bounded above by the number of lattice points $\ZZ^{d+1} M$ in the bounded set $\Delta\scrD \times [0,\theta^{d+1}]$, where $\theta^{d+1}$ is the uniform upper bound from Proposition \ref{prop:twoB}. In view of Mahler's criterion, the number of lattice points $\ZZ^{d+1} M$ in any fixed bounded set is bounded above uniformly for all $\Gamma M$ in a given compact subset of $\GamG$ (which here is $\scrC D(\theta)^{-1}$). This proves the claim.
\end{proof}

\begin{lem}\label{newlyn}
Let $\scrD$ be bounded and convex with non-empty interior. There is a point $P\in\partial(\Delta\scrD)$ such that, for every point $Q$ on the open line segment $\overline{0P}$, there exists a $t\in\partial\scrD$ such that $Q=\overline{0P}\cap (\partial\scrD-t)$.
\end{lem}

\begin{proof}
First of all, let us establish that there is a unit vector $u\in\RR^d$ with the property that, for all $R,S\in\partial\scrD$ with the open line segment $\overline{RS}$ parallel to $u$, we have that $\overline{RS}\subset\scrD^\circ$. To this end, observe that if $\overline{RS}\not\subset\scrD^\circ$ for some $R,S\in\partial\scrD$, then by convexity $\overline{RS}\subset\partial\scrD$. The set of unit vectors that are parallel to line segments in $\partial\scrD$ has finite $(d-2)$-dimensional Haussdorff measure \cite[Theorem 1]{Ewald70}. Any unit vector $u$ in the complement of that set will thus have the required property.

Now take $u$ as above, and let $\lambda$ be the length of the longest line segment parallel to $u$, with endpoints in $\partial\scrD$. We claim that the conclusion of the lemma will then be satisfied by choosing $P\in\partial(\Delta\scrD)$ so that $\overline{0P}$ is parallel to $u$ and has length $\lambda$. To see why this is true, suppose that $Q$ is on the open line segment $\overline{0P}$ and that the length of $\overline{0Q}$ is $\ell\in(0,\lambda]$ (see Figure 1). Then, by our choice of $u$, there are points $R,S\in\partial\scrD$ such that the line segment $\overline{RS}$ is parallel to $u$ and has length $\ell$. This follows from the facts that: (i) $|\overline{RS}|$ is a continuous function of $R\in\partial\scrD$, since $\overline{RS}\subset\scrD^\circ$, and (ii) by convexity, there is an $R_0\in\partial D$ so that $(R_0+u\RR) \cap \scrD^\circ=\emptyset$ and hence $\lim_{R\to R_0}\ell(R)=0$. It is now clear that the conclusion of the lemma is satisfied by taking $t=R$.
\end{proof}

\begin{figure}[h]
\centering
\def\svgwidth{0.8\columnwidth}
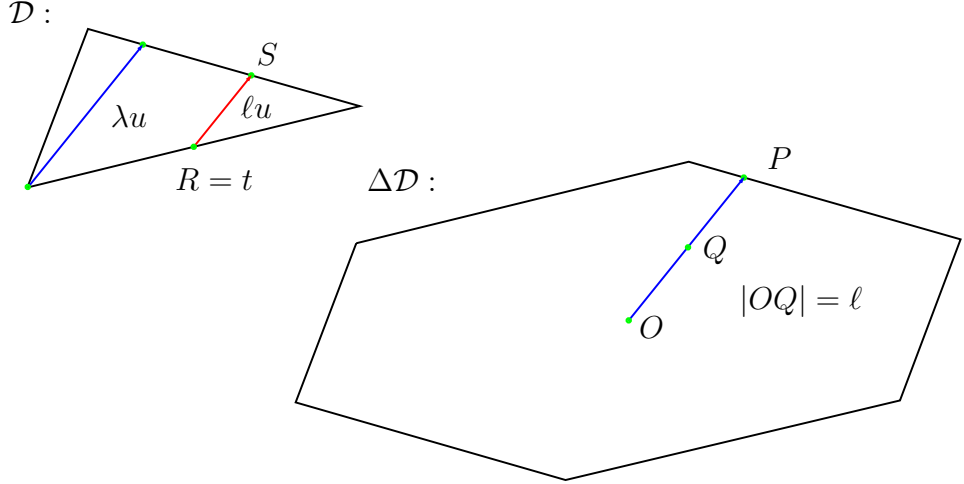
\caption{Pictorial description of proof of Lemma \ref{newlyn}}
\end{figure}


\begin{prop}\label{prop:unbddNEW}
Let $\scrD$ be bounded and convex with non-empty interior. Then there exist $\lambda>0$, $\epsilon_0>0$, such that for every $\epsilon\in(0,\epsilon_0]$ there exist non-empty open sets $\scrU\subset\GamG$ and $\scrV_m\subset\scrD^\circ$, $m=1,\ldots,\lfloor \epsilon^{-1}\lambda\rfloor$, such that
\begin{equation}
F(M,t)=\epsilon\, (m + a_m(M))
\end{equation}
for $\Gamma M\in\scrU$ and $t\in \scrV_m$, with $a_m(M)\in[-\frac{1}{100},\frac{1}{100}]$.
\end{prop}

\begin{proof}
Let $u, \lambda,$ and $P$ (the end point of $\lambda u$) be as in the proof of Lemma \ref{newlyn}, choose $\epsilon>0$, and define vectors $v_0=(\epsilon u, -\epsilon)$ and $v_1=(\lambda u,0)$ in $\RR^{d+1}$. Suppose that $u,u_2,\ldots,u_d$ is an orthonormal basis for $\RR^d$ with respect to the standard Euclidean metric, and define the matrix
\begin{equation}
M_\epsilon=
\begin{pmatrix}
\epsilon u & -\epsilon \\
\lambda u & 0 \\
(\epsilon \lambda)^{-\frac{1}{d-1}} u_2 & 0\\
\vdots & \vdots \\
(\epsilon \lambda)^{-\frac{1}{d-1}} u_d & 0
\end{pmatrix} \in G.
\end{equation}

{\em Step 1.} Our first aim is to show that there is $\epsilon_0\in(0,\lambda]$, such that for every $\epsilon\in(0,\epsilon_0]$ and positive integer $m \leq\lfloor\epsilon^{-1}\lambda\rfloor$, there is $t_m\in\scrD^\circ$ such that $F(M_\epsilon,t_m)=m\epsilon$. Note that the row vectors $v_0,v_1,\ldots,v_d$ of $M_\epsilon$ form a basis of the unimodular lattice $\ZZ^{d+1}M_\epsilon$. A general vector in this lattice is of the form
$v=a_0 v_0 +\ldots+a_d v_d$ with $a_i\in\ZZ$. If at least one of $a_2,\ldots,a_d$ is non-zero then, for sufficiently small $\epsilon_0>0$, we have $v\notin\Delta\scrD\times\RR$, and hence the corresponding lattice vector $v$ will not contribute to $F(M_\epsilon,t)$, for any $t\in\scrD$.

\begin{figure}[h]
\centering
\def\svgwidth{0.6\columnwidth}
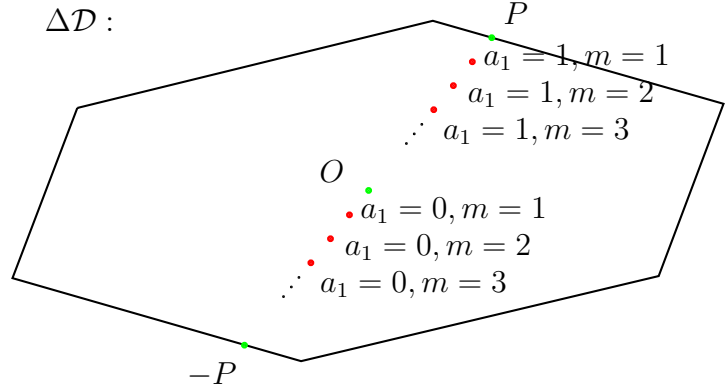
\caption{Positions of relevant points of $\ZZ^{d+1}M_\epsilon$}
\end{figure}

We can therefore restrict our attention to vectors $v$ with coefficients $a_2=\ldots=a_d=0$; that is $v=(a_0 \epsilon u + a_1 \lambda u, -a_0\epsilon)$. Only vectors whose last coordinate is positive contribute. Hence, with $n=-a_0$, we have for any $t\in\scrD$
\begin{equation}\label{Fdef001}
F(M_\epsilon,t)=\epsilon \min\big\{ n\in\NN^* \;\big|\; (a_1 \lambda -\epsilon n) u\in\scrD-t~\text{for some}~ a_1\in\ZZ \big\} .
\end{equation}
Given $Q\in\overline{0P}$, Lemma \ref{newlyn} guarantees the existence of $t'\in\partial\scrD$ such that $Q=\overline{0P}\cap (\partial\scrD-t')$. Take $Q=Q_m=(\lambda -\epsilon m +\tfrac12\epsilon) u$ and denote the corresponding $t'$ by $t_m'$. Note that $Q_m\in\overline{0P}$ since $1\leq m\leq \lfloor\epsilon^{-1}\lambda\rfloor$ implies  $0<\lambda-\epsilon m+\tfrac12\epsilon< \lambda$. Thus $0,Q_m\in\partial\scrD-t_m'$.
Define the interval $\scrJ_m\subset[0,\lambda]$ by $\scrJ_m u = (\scrD-t_m') \cap [0,\lambda] u$. By construction, the interval $\scrJ_m$ is equal to $[0,\lambda-\epsilon m +\tfrac12\epsilon]$ or one of its open/half-open variants, depending on $\scrD$.
Set $t_m=t_m'+\tfrac14\epsilon u$. This ensures $0\in\scrD^\circ-t_m$ and hence $t_m\in\scrD^\circ$.
Then
\begin{equation}\label{Fdef002}
F(M_\epsilon,t_m)=\epsilon \min\big\{ n\in\NN^* \;\big|\; a_1 \lambda -\epsilon n+\tfrac14\epsilon \in\scrJ_m ~\text{for some}~ a_1\in\ZZ \big\} .
\end{equation}
In case $\scrJ_m$ is closed, $a_1 \lambda -\epsilon n +\tfrac14\epsilon \in\scrJ_m$ is equivalent to
\begin{equation}
0\leq a_1 \lambda -\epsilon n +\tfrac14\epsilon \leq \lambda-\epsilon m +\tfrac12\epsilon.
\end{equation}
The first inequality yields $a_1 \lambda \geq  \epsilon n-\tfrac14\epsilon$, which is positive, and hence $a_1\in\NN^*$. The second inequality yields
\begin{equation}
n \geq (a_1-1) \epsilon^{-1} \lambda + m - \tfrac14  .
\end{equation}
The smallest $n\in\NN^*$ satisfying this inequality for any $a_1\in\NN^*$ is $n=m$ and occurs for $a_1=1$. This choice of $a_1$ is consistent with the first inequality as long as $m=n\leq \lfloor\epsilon^{-1} \lambda\rfloor$, as assumed. The same argument goes through, with the same result, in the remaining cases when $\scrJ_m$ is open or half-open. We conclude that $F(M_\epsilon,t_m)=m\epsilon$ for $m=1,\ldots,\lfloor \epsilon^{-1}\lambda\rfloor$.

Figures 2 and 3 provide visual explanations of these arguments.

\vspace*{.1in}

\begin{figure}[h]
\centering
\def\svgwidth{0.8\columnwidth}
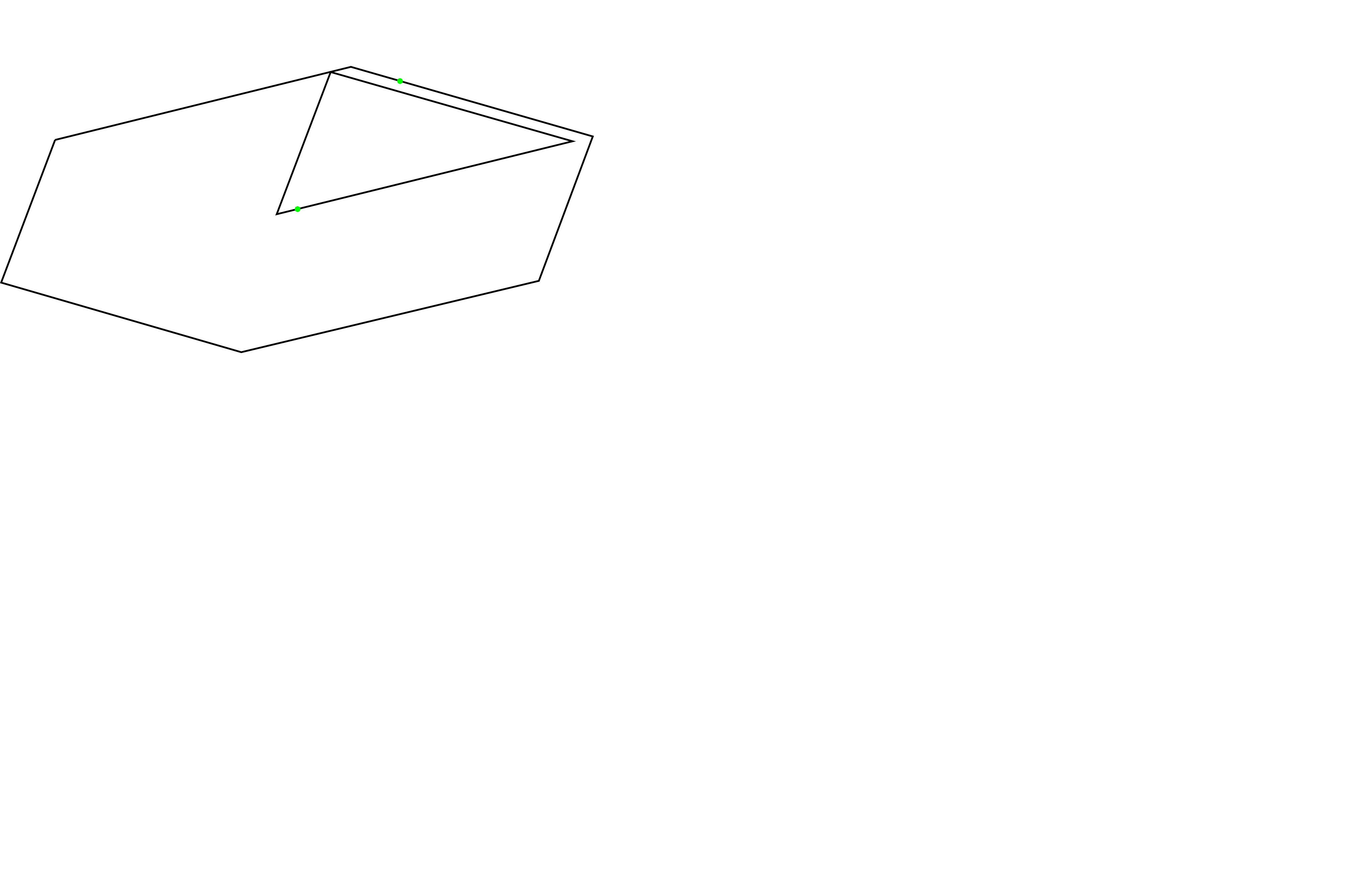
\caption{Choices of $t$ for which $F(M_\epsilon,t)$ takes different values}
\end{figure}

\vspace*{.1in}

{\em Step 2.} The fact that the argument of Step 1 is independent of whether or not we have strict inequalities,  also manifests itself in the continuity of $F$ at $(\Gamma M_\epsilon,t_m)$, which we will establish now. We assume $\epsilon,m$ are as in Step 1.
In view of Proposition \ref{prop:cont} it is sufficient to check that
\begin{equation}\label{exset01}
(\ZZ^{d+1} M_\epsilon\setminus\{0\})
\cap \partial((\scrD - t_m)\times [0,\kappa_m]) = \emptyset ,
\end{equation}
for a given choice of $\kappa_m>\epsilon m$; let us take $\kappa_m=\epsilon m+\tfrac15\epsilon$.
For $\epsilon\leq \epsilon_0$ as in Step 1, this is equivalent to
\begin{equation}
\{ (a_1 \lambda-n \epsilon )  u, n\epsilon) \mid (n,a_1)\in\ZZ^2\setminus\{0\} \} \cap \partial((\scrD - t_m)\times [0,\kappa_m]) = \emptyset .
\end{equation}
This in turn is equivalent to the conditions that
\begin{equation}\label{check001}
\{ (a_1 \lambda -\epsilon n+\tfrac14\epsilon, n\epsilon) \in \cl \scrJ_m\times \{0,\kappa_m\} \mid (n,a_1)\in\ZZ^2\setminus\{0\}  \} =\emptyset
\end{equation}
and that
\begin{equation}\label{check002}
\{ (a_1 \lambda -\epsilon n+\tfrac14\epsilon, n\epsilon) \in \partial\scrJ_m \times [0,\kappa_m] \mid (n,a_1)\in\ZZ^2\setminus\{0\}  \} =\emptyset ,
\end{equation}
where $\partial\scrJ_m=\{0,\lambda-\epsilon m +\tfrac12\epsilon\}$.

Condition \eqref{check001} is equivalent to the statement that
\begin{equation}
0\leq a_1 \lambda +\tfrac14\epsilon \leq \lambda-\epsilon m +\tfrac12\epsilon
\end{equation}
has no solution for $a_1\in\ZZ\setminus\{0\}$. This is indeed the case: the first inequality rules out negative $a_1$, the second positive $a_1$ (since $m\geq 1$).

Condition \eqref{check002} is equivalent to the conditions that
\begin{equation}\label{check003}
 [ a_1 \epsilon^{-1} \lambda -n +\tfrac14 = 0  , \; 0\leq n\leq m+\tfrac15 ]
\end{equation}
has no solution for any $(n,a_1)\in\ZZ^2\setminus\{0\}$, and that
\begin{equation}\label{check004}
[ a_1 \epsilon^{-1}\lambda -n +\tfrac14 = \epsilon^{-1}\lambda- m +\tfrac12 , \; 0\leq n\leq m+\tfrac15 ]
\end{equation}
has no solution for any $(n,a_1)\in\ZZ^2\setminus\{0\}$. As to \eqref{check003}, $a_1\leq 0$ does not yield a solution since $n$ is a non-negative integer and $\epsilon^{-1} \lambda\geq 1$; $a_1\geq 1$ on the other hand cannot lead to a solution since $n\leq m\leq\lfloor\epsilon^{-1}\lambda\rfloor$. Finally, \eqref{check004} can be written as
\begin{equation}\label{check004b}
[ (a_1-1)\epsilon^{-1} \lambda + m-n =  \tfrac14 , \; 0\leq n\leq m ].
\end{equation}
Solutions with $a_1\geq 1$ are not possible since $m-n$ is a non-negative integer and $\epsilon^{-1} \lambda\geq 1$; and $a_1\leq 0$ can be ruled out since $m-n\leq m\leq\lfloor\epsilon^{-1}\lambda\rfloor$. This proves $F$ is continuous at $(\Gamma M_\epsilon,t_m)$.

{\em Step 3.} We conclude by noting that continuity at $(\Gamma M_\epsilon,t_m)$ implies the proposition. The reason why $a_m(M)$ is independent of $t$ is that the function $t\mapsto F(M, t)$ is piecewise constant.
\end{proof}

The following is an immediate consequence of Proposition \ref{prop:unbddNEW}.

\begin{cor}\label{cor:dense}
Let $\scrD$ be bounded and convex with non-empty interior. Let $M_1,M_2,\ldots\in G$, and $R_1<R_2<\cdots \to\infty$. If $(\Gamma M_i)_{i\in\NN^*}$ is dense in $\GamG$, then
\begin{equation}\sup_{i\in\NN^*} \scrG_{R_i}( M_i)=\infty.
\end{equation}
\end{cor}
\begin{proof}
Given any $\epsilon>0$ choose $\scrU\subset\GamG$ and $\scrV_m\subset\scrD^\circ$ as in Proposition  \ref{prop:unbddNEW}. Since $(\Gamma M_i)_{i\in\NN^*}$ is dense in $\GamG$, there exist $i\in\NN^*$ and $k_m\in R_i\scrD\cap\ZZ^d$ such that (i) $\Gamma M_i\in\scrU$ and (ii) $R_i^{-1} k_m \in\scrV_m$ for all $m=1,\ldots,\lfloor \lambda \epsilon^{-1}\rfloor$. Thus $\scrG_{R_i}( M_i)\geq \lfloor \lambda \epsilon^{-1}\rfloor$.
\end{proof}

Corollary \ref{cor:dense} is a crucial ingredient in the proof of Theorem \ref{thm:one} below.

\section{Dynamics of diagonal actions; proofs of Theorems \ref{thm:one}--\ref{thm:three}} \label{sec:three}

Let
\begin{equation}\label{phi-def}
\Phi^s =
\begin{pmatrix}
\e^{-s} 1_d & \trans 0 \\ 0 & \e^{ds}
\end{pmatrix}
\in G.
\end{equation}
The right translation
\begin{equation}
\GamG \to \GamG, \qquad \Gamma M \mapsto \Gamma M \Phi^s
\end{equation}
generates a flow on $\GamG$ which is well-known to be ergodic with respect to the unique $G$-invariant probability measure $\mu$ on $\GamG$. A standard argument (see e.g.~\cite[Cor.~3.7]{Mark2000}) shows that there is a set of full Lebesgue measure $P\subset\RR^d$ such that for $\alpha\in P$, the orbit
\begin{equation}\label{orbit}
\bigg\{ \Gamma \begin{pmatrix} 1_d & \trans\alpha \\ 0 & 1 \end{pmatrix} \Phi^s \;\bigg|\; s\in\RR_{\geq 0}\bigg\}
\end{equation}
is dense in $\GamG$. This in turn implies \cite[Cor.~3.8]{Mark2000} that
\begin{equation}\label{friday}
\bigg\{ \Gamma \begin{pmatrix} 1_d & \trans\alpha \\ 0 & 1 \end{pmatrix} \Phi^{s_i} \;\bigg|\; i\in\NN^* \bigg\}
\end{equation}
is dense in $\GamG$, provided $s_1<s_2<\ldots \to\infty$ such that $s_{i+1}-s_i\to 0$. In view of Corollary \ref{cor:dense},
this establishes the first claim \eqref{diverge} of Theorem \ref{thm:one} (take $s_i=\log R_i$). As to the second claim \eqref{bdd}: the denseness of the orbit \eqref{friday} implies that it returns to a compact set infinitely often. Hence  \eqref{bdd} follows from Proposition \ref{prop:three}, and the proof of Theorem \ref{thm:one} is complete.

Dani's correspondence \cite{Dani1985} states that the orbit \eqref{orbit} is bounded if and only if $\alpha$ is badly approximable. Thus there is a compact $\scrC\subset\GamG$ which contains \eqref{orbit}. This means that for all $s\geq\log \theta$ we have
\begin{equation}
\Gamma \begin{pmatrix} 1_d & \trans\alpha \\ 0 & 1 \end{pmatrix} \Phi^s \in \scrC D(\theta)^{-1}.
\end{equation}
Proposition \ref{prop:three} then implies that
\begin{equation}
\scrG\bigg( \begin{pmatrix} 1_d & \trans\alpha \\ 0 & 1 \end{pmatrix} \Phi^s \bigg)\leq C_\theta
\end{equation}
for all $s\geq\log \theta$.
Thus $G(\alpha,R\scrD) \leq C_\theta$ for all $R\geq \theta$. Finally, $G(\alpha,R\scrD)$ is trivially bounded by the number of points in $\ZZ^d\cap R\scrD$, which in turn is uniformly bounded for all $R\leq \theta$.
This yields Theorem \ref{thm:two}.

Our proof of Theorem \ref{thm:three} is similar, but slightly more complicated. The plan is to assume
\begin{equation}\label{eq:littlex}
\liminf_{n\to\infty} n \| n \alpha_1 \|_{\RR/\ZZ} \cdots \| n \alpha_d\|_{\RR/\ZZ} > 0
\end{equation}
and show that this contradicts the hypothesis \eqref{eq:little2} of Theorem \ref{thm:three}.

By the well known correspondence of the Littlewood conjecture and unbounded orbits (cf.~\cite[Prop.~11.1]{EinsKatoLind2006}), we have that \eqref{eq:littlex} implies that there is a compact set $\scrC'\subset\GamG$ such that for $T=\diag(T_1,\ldots,T_d)$
\begin{equation}\label{orbit001}
\bigg\{ \Gamma \begin{pmatrix} 1_d & \trans 0 \\ -\alpha & 1 \end{pmatrix} \begin{pmatrix} T & \trans 0 \\ 0 & (\det T)^{-1} \end{pmatrix} \;\bigg|\; T_1,\ldots,T_d\geq 1 \bigg\} \subset\scrC'.
\end{equation}
Taking the transpose inverse, we infer that
\begin{equation}\label{orbit002}
\bigg\{ \Gamma \begin{pmatrix} 1_d & \trans\alpha \\ 0 & 1 \end{pmatrix} \begin{pmatrix} T^{-1} & \trans 0 \\ 0 & \det T \end{pmatrix} \; \bigg|\; T_1,\ldots,T_d\geq 1 \bigg\} \subset\scrC ,
\end{equation}
where $\scrC=\{\Gamma \trans M^{-1} \mid \Gamma M\in\scrC'\}\subset \GamG$ is compact.

In view of Proposition \ref{prop:three}, we have
\begin{equation}\label{eqn.GapsBdd1}
\scrG\bigg( \begin{pmatrix} 1_d & \trans\alpha \\ 0 & 1 \end{pmatrix} \begin{pmatrix} T^{-1} & \trans 0 \\ 0 & \det T \end{pmatrix} D(\theta)^{-1} \bigg) \leq C_\theta
\end{equation}
for all $T_1,\ldots,T_d\geq 1$ and for all $\theta>\max\{1,\overline\rho(\scrC,\scrD\times (0,1])\}$. In other words,
\begin{equation}\label{three-eleven}
G(\alpha,\mathcal{D}_T)\leq C_\theta
\end{equation}
for all $T_1,\ldots,T_d\geq \theta$. To establish a contradiction with hypothesis \eqref{eq:little2} of Theorem \ref{thm:three}, what needs to be shown is that \eqref{three-eleven} in fact holds for all $T_1,\ldots,T_d\geq 1$. The key point in achieving this is the following lemma.

\begin{lem}\label{lem:last}
Let $\scrD$ be as in Theorem \ref{thm:three}. If \eqref{eq:littlex} holds, then there is a constant $\Theta<\infty$ such that
\begin{equation}\label{lastlemeq}
\max_{k\in\scrD_T\cap \ZZ^d} F\bigg( \begin{pmatrix} 1_d & \trans\alpha \\ 0 & 1 \end{pmatrix} \begin{pmatrix} T^{-1} & \trans 0 \\ 0 & \det T \end{pmatrix} , k T^{-1} \bigg)\leq \Theta
\end{equation}
for all $T_1,\ldots,T_d\geq 1$.
\end{lem}

For $T$ as above, denote by $\scrG_T(M)$ the number of distinct values of $F(M,k T^{-1})$ as $k\in\ZZ^d$ runs over $\scrD T$.
Since the orbit \eqref{orbit002} is contained in a compact set, once Lemma \ref{lem:last} has been proved we may conclude (by the same argument as in the proof of Proposition \ref{prop:three}, with $\scrG$ replaced by $\scrG_T$) that there is $C_1<\infty$ such that
\begin{equation}\label{eqn.GapsBdd1001}
\scrG_T\bigg(  \begin{pmatrix} 1_d & \trans\alpha \\ 0 & 1 \end{pmatrix} \begin{pmatrix} T^{-1} & \trans 0 \\ 0 & \det T \end{pmatrix}  \bigg) \leq C_1,
\end{equation}
for all $T_1,\ldots,T_d\geq 1$. Since the left hand side of \eqref{eqn.GapsBdd1001} is equal to $G(\alpha,\mathcal{D}_T)$, this completes the proof of  Theorem \ref{thm:three}.

\begin{proof}[Proof of Lemma \ref{lem:last}]
It is sufficient to show that there is $\theta\geq 1$ such that, for every (possibly empty) subset $\mathcal{I}\subset\{1,\ldots ,d\}$, eq.~\eqref{lastlemeq} holds for all $T$ with $T_i\geq \theta$ ($i\in\scrI$) and $1\leq T_i<\theta$ ($i\notin\scrI$). We assume first $\scrI\neq\emptyset$.

Let us highlight the dependence on $\scrD$ and dimension $d$ by writing $F_\scrD^{d}(M,t)=F(M,t)$.
We denote by $\alpha_\scrI\in\RR^{|\scrI |}$ and $t_\scrI\in\RR^{|\scrI |}$ the orthogonal projections of $\alpha$ and $t$, respectively, onto the subspace corresponding the the coordinates indexed by $\scrI$, and denote by $T_\scrI$ the diagonal matrix with entries $T_i$ ($i\in\scrI$). Let $G_\scrI=\SL(|\scrI |+1,\RR)$ and $\Gamma_\scrI=\SL(|\scrI |+1,\ZZ)$.

Set $\scrQ=[0,\epsilon)$ and $\scrQ_T^d=[0,\epsilon T_1)\times\cdots\times[0,\epsilon T_d)$. Note that
\begin{equation}
S(\alpha_\scrI,\scrQ_T^{|\scrI |}) = \bigg\{ \sum_{i\in\scrI} m_i \alpha_i \bmod 1\,\bigg|\, m_i\in\ZZ \cap [0,\epsilon T_i) \; (i\in\scrI)\bigg\}  ,
\end{equation}
\begin{equation}
S(\alpha,\scrQ_T^d) = \bigg\{ \sum_{i=1}^d m_i \alpha_i  \bmod 1\,\bigg|\, m_i\in\ZZ \cap [0,\epsilon T_i)  \; (i=1,\ldots,d) \bigg\}  ,
\end{equation}
and hence
\begin{equation}
\emptyset\neq S(\alpha,\scrQ^{|\scrI |}_T)\subset S(\alpha,\scrQ_T^d) \subset S(\alpha,\scrD_T) ,
\end{equation}
since $\scrQ^d\subset\scrD$ by assumption.
Since removing elements from a set does not decrease the size of gaps in the set, we have that the maximal gap in $S(\alpha,\scrD_T)$ is bounded above by the maximal gap in $S(\alpha,\scrQ^{|\scrI |}_T)$.
Therefore, in view of \eqref{key} and $\prod_{\i\notin\scrI} T_i \leq \theta^{d-|\scrI |}$, we have
\begin{multline}
\max_{k\in\scrD_T \cap \ZZ^d} F_{\scrD}^d\bigg(  \begin{pmatrix} 1_d & \trans\alpha \\ 0 & 1 \end{pmatrix} \begin{pmatrix} T^{-1} & \trans 0 \\ 0 & \det T \end{pmatrix} , k T^{-1} \bigg) \\
 \leq \theta^{d-|\scrI |} \max_{k\in\scrQ^{|\scrI |}_T\cap \ZZ^d} F_{\scrQ^{|\scrI |}}^{|\scrI|}\bigg( \begin{pmatrix} 1_{|\scrI |} & \trans\alpha_\scrI \\ 0 & 1 \end{pmatrix} \begin{pmatrix} T_\scrI^{-1} & \trans 0 \\ 0 & \det T_\scrI \end{pmatrix} , k T_\scrI^{-1} \bigg) ,
\end{multline}
for all $T$ with $T_i\geq \theta$ ($i\in\scrI$) and $1\leq T_i<\theta$ ($i\notin\scrI$).
Our assumption \eqref{eq:littlex} implies that
\begin{equation}
  \liminf_{n\rar\infty} n\prod_{i\in\mathcal{I}}\|n\alpha_i\|>0,
\end{equation}
and hence (by the same argument leading to \eqref{orbit002})
\begin{equation}\label{orbit002BB}
\bigg\{ \Gamma_\scrI \begin{pmatrix} 1_{|\scrI |} & \trans\alpha_\scrI \\ 0 & 1 \end{pmatrix} \begin{pmatrix} T_\scrI^{-1} & \trans 0 \\ 0 & \det T_\scrI \end{pmatrix} \; \bigg|\; T_i \geq 1 \; (i\in\scrI) \bigg\} \subset \scrC_\scrI
\end{equation}
for some compact $\scrC_\scrI\subset\Gamma_\scrI\backslash G_\scrI$.
Proposition \ref{prop:twoB} now tells us that, for any $\theta>\overline\rho(\scrC_\scrI,\scrQ^{|\scrI |})$,
\begin{equation}
\sup_{t_\scrI\in\scrQ^{|\scrI |}} F_{\scrQ^{|\scrI |}}^{|\scrI|}\bigg( \begin{pmatrix} 1_{|\scrI |} & \trans\alpha_\scrI \\ 0 & 1 \end{pmatrix} \begin{pmatrix} T_\scrI^{-1} & \trans 0 \\ 0 & \det T_\scrI \end{pmatrix} , t_\scrI \bigg) \leq \theta^{|\scrI |+1}
\end{equation}
for $T_i\geq \theta$ ($i\in\scrI$). Therefore, for all $\theta>\max_{\scrI\neq\emptyset} \overline\rho(\scrC_\scrI,\scrQ^{|\scrI |})$, we have
\begin{equation}\label{lastlemeqBB}
\max_{k\in\scrD_T \cap \ZZ^d} F\bigg(  \begin{pmatrix} 1_d & \trans\alpha \\ 0 & 1 \end{pmatrix} \begin{pmatrix} T^{-1} & \trans 0 \\ 0 & \det T \end{pmatrix} , kT^{-1} \bigg)\leq \theta^{d+1}
\end{equation}
for all $T_i\geq 1$ with $\max\{T_1,\ldots,T_d\} \geq \theta$. The remaining case $\scrI=\emptyset$, where $1\leq T_1,\ldots,T_d < \theta$, is immediate via \eqref{key}, since the maximal gap is bounded by $1$, and the determinant is bounded above by $\theta^d$ for this collection of $T_i$'s.
\end{proof}

\section{Proof of Theorem \ref{thm:five}} \label{sec:four}

For $d=1$ the statement of Theorem \ref{thm:five} is obviously implied by the three gap theorem, so assume without loss of generality that $d\ge 2$. Let $\mathcal{D}=[0,1)^d$ as in the statement of the theorem and note that, in order to prove \eqref{eqn.SupFinite}, it is enough to consider the case when $T_i=M_i\in\NN^*$ for $1\le i\le d$.

In our proof we are going to use a theorem due to Chevallier \cite[Theorem 1]{Chev2000}, which is a higher dimensional version of Geelen and Simpson's result from \cite{GeelSimp1993}. We can express Chevallier's result in our language as the statement that, for any $\alpha\in\R^d,$ and for any $N_1,\ldots ,N_d\in\NN^*$, if $\mathcal{B}=[0,N_1)\times\cdots\times[0,N_d)\subset\R^d$ then
\begin{equation}\label{eqn.Chev}
G(\alpha, \mathcal{B})\le \prod_{i=1}^{d-1}N_i+3\prod_{i=1}^{d-2}N_i+1.
\end{equation}
Under the hypothesis of Theorem \ref{thm:five}, we can find integers $Q$ and $B_i,~1\le i\le d$, such that
\begin{equation}\label{eqn.AlphBetaRelation}
Q\alpha_i-B_i\beta\in\Z\quad\text{for}\quad 1\le i\le d.
\end{equation}
By replacing $\beta$ with an integer multiple of $\beta$, we may assume without loss of generality that $\mathrm{gcd}(B_1,\ldots B_d)=1$. Let us restrict our attention to the situation when $B_i>0$ for each $i$. If this is not the case then the proof follows by minor modifications of the argument we are about to give.

Suppose that $T_i=M_i\in\NN^*$ and for each $i$ let $A_i$ and $R_i$ be the unique integers for which $A_i\ge 0,~1\le R_i\le Q$, and
\begin{equation}
M_i=A_iQ+R_i.
\end{equation}
Then we have that
\begin{align}
S(\alpha,\mathcal{D}_T)=\bigcup_{\mathcal{I}\subset\{1,\ldots ,d\}}S_\mathcal{I},
\end{align}
where $\mathcal{I}$ runs over all subsets of $\{1,\ldots ,d\}$ (including the empty set) and $S_\mathcal{I}$ is defined by
\begin{align}
  S_\mathcal{I}&=\left\{\sum_{i\in\mathcal{I}}(A_iQ+r_i)\alpha_i+\sum_{i\notin\mathcal{I}}(a_iQ+r_i)\alpha_i~\mathrm{mod}~1~:~\substack{1\le r_i\le R_i,i\in\mathcal{I}\\ 0\le a_i< A_i,1\le r_i\le Q, i\notin\mathcal{I}}\right\} \\
  &=\left\{\left(\sum_{i\in\mathcal{I}}A_iB_i+\sum_{i\notin\mathcal{I}}a_iB_i\right)\beta+\sum_{i=1}^dr_i\alpha_i~\mathrm{mod}~1~:~\substack{1\le r_i\le R_i,i\in\mathcal{I}\\ 0\le a_i< A_i,1\le r_i\le Q, i\notin\mathcal{I}}\right\}.\label{eqn.S_IDesc}
\end{align}
Now we will need the following elementary number theoretic lemma.
\begin{lem}\label{lem.ElemNumThy}
  Suppose that $k\ge 2$ is an integer, that $q_1,\ldots ,q_k\in\NN^*$, and that, for each $1\le i\le k$, $C_i$ and $D_i$ are integers satisfying
  \begin{equation}
  D_i-C_i\ge\max_{1\le j\le k}q_j.
  \end{equation}
  Let $r=\mathrm{gcd}(q_1,\ldots ,q_k)$ and
  \begin{equation}
  \mathcal{A}=\left\{\sum_{i=1}^ka_iq_i~:~C_i\le a_i\le D_i\right\},
  \end{equation}
  and set
  \begin{equation}
  C=\sum_{i=1}^dC_iq_i\quad\text{and}\quad D=\sum_{i=1}^dD_iq_i.
  \end{equation}
  Then we have that
  \begin{equation}\label{eqn.LemInclusion1}
    \mathcal{A}\subset\left\{mr~:~\frac{C}{r}\le m\le\frac{D}{r}\right\},
  \end{equation}
  and
  \begin{equation}\label{eqn.LemInclusion2}
    \left\{mr~:~\frac{C}{r}+\frac{1}{r^2}\sum_{i=1}^{k-1}q_iq_{i+1}\le m\le\frac{D}{r}-\frac{1}{r^2}\sum_{i=1}^{k-1}q_iq_{i+1}\right\}\subset\mathcal{A}.
  \end{equation}
\end{lem}
\begin{proof}
  The inclusion in equation \eqref{eqn.LemInclusion1} is quite obvious, so we will focus on proving \eqref{eqn.LemInclusion2}. Our proof is by induction on $k$, so first let us consider the case when $k=2$. In this case, if $n\in\NN^*$ and if there is an integer solution $(a_1,a_2)$ to the equation
  \begin{equation}
  a_1q_1+a_2q_2=n,
  \end{equation}
  then it must be the case that $n=mr$ for some $m\in\Z$. Then we have that
  \begin{equation}
  a_2=m(q_2/r)^{-1}~\mathrm{mod}~(q_1/r)\quad\text{and}\quad a_1=\frac{n-a_2q_2}{q_1}.
  \end{equation}
  We are imposing the conditions that $C_i\le a_i\le D_i$, and the assumption that $D_2-C_2\ge q_1$ guarantees that there is at least one choice of $a_2$ satisfying the first equation here. The smallest admissible choice for such an integer $a_2$ is at least as small as $C_2+q_1/r$, and the largest admissible choice for such an $a_2$ is at least as large as $D_2-q_1/r$. As long as there is at least one admissible choice of $a_1$, as $a_2$ runs over this range, then we can guarantee that $n\in\mathcal{A}$. This will be the case if
  \begin{equation}
  \frac{n-(C_2+q_1/r)q_2}{q_1}\ge C_1\qquad\text{and}\qquad \frac{n-(D_2-q_1/r)q_2}{q_1}\le D_1,
  \end{equation}
  and these inequalities will both be satisfied if
  \begin{equation}
  \frac{C}{r}+\frac{q_1q_2}{r^2}\le m\le \frac{D}{r}-\frac{q_1q_2}{r^2}.
  \end{equation}
  This finishes the proof when $k=2$.

  Now suppose that $k>2$ and that the lemma is true, for all choices of parameters, with $k$ replaced by $k-1$.
  Let $r'=\mathrm{gcd}(q_1,\ldots ,q_{k-1})$,
  \begin{equation}
  C'=\sum_{i=1}^{k-1}C_iq_i,\qquad\text{and}\qquad D'=\sum_{i=1}^{k-1}D_iq_i,
  \end{equation}
  and set
  \begin{equation}
  \mathcal{A}'=\left\{\sum_{i=1}^{k-1}a_iq_i~:~C_i\le a_i\le D_i\right\}.
  \end{equation}
  Then it is clear that
  \begin{equation}
  \mathcal{A}=\left\{n+a_kq_k~:~n\in\mathcal{A}', C_k\le a_k\le D_k\right\}
  \end{equation}
  and, by our inductive hypothesis, we have that $\mathcal{A}$ contains the set
  \begin{equation}\label{eqn.LemGenkSet}
  \left\{mr'+a_kq_k~:~\frac{C'}{r'}+\frac{1}{{r'}^2}\sum_{i=1}^{k-2}q_iq_{i+1}\le m\le\frac{D'}{r'}-\frac{1}{{r'}^2}\sum_{i=1}^{k-2}q_iq_{i+1},~C_k\le a_k\le D_k\right\}.
  \end{equation}
  Now let $\tilde{q}_1=r',~\tilde{q}_2=q_k,$
  \begin{align}
  \tilde{C}_1&=\frac{C'}{r'}+\frac{1}{{r'}^2}\sum_{i=1}^{k-2}q_iq_{i+1},\\
  \tilde{D}_1&=\frac{D'}{r'}-\frac{1}{{r'}^2}\sum_{i=1}^{k-2}q_iq_{i+1},
  \end{align}
  $\tilde{C}_2=C_k$, and $\tilde{D}_2=D_k$. Then $\mathrm{gcd}(\tilde{q}_1,\tilde{q}_2)=r$ and, by the same argument used above to settle the $k=2$ case, we find that the set \eqref{eqn.LemGenkSet} contains all integers of the form $mr$, with
  \begin{equation}
    \frac{\tilde{C}_1\tilde{q}_1+\tilde{C}_2\tilde{q}_2}{r}+\frac{\tilde{q}_1\tilde{q}_2}{r^2}\le m\le \frac{\tilde{D}_1\tilde{q}_1+\tilde{D}_2\tilde{q}_2}{r}-\frac{\tilde{q}_1\tilde{q}_2}{r^2}.
  \end{equation}
  Finally, we compute that
  \begin{align}
    \frac{\tilde{C}_1\tilde{q}_1+\tilde{C}_2\tilde{q}_2}{r}+\frac{\tilde{q}_1\tilde{q}_2}{r^2}&=\frac{1}{r}\sum_{i=1}^kC_iq_i+\frac{1}{rr'}\sum_{i=1}^{k-2}q_iq_{i+1}+\frac{r'q_k}{r^2}\\
    &\le\frac{C}{r}+\frac{1}{r^2}\sum_{i=1}^{k-1}q_iq_{i+1},
  \end{align}
  and that
  \begin{align}
    \frac{\tilde{D}_1\tilde{q}_1+\tilde{D}_2\tilde{q}_2}{r}-\frac{\tilde{q}_1\tilde{q}_2}{r^2}&=\frac{1}{r}\sum_{i=1}^kD_iq_i-\frac{1}{rr'}\sum_{i=1}^{k-2}q_iq_{i+1}-\frac{r'q_k}{r^2}\\
    &\ge\frac{D}{r}-\frac{1}{r^2}\sum_{i=1}^{k-1}q_iq_{i+1}.
  \end{align}
  It is clear from this that \eqref{eqn.LemInclusion2} holds, and our inductive argument is complete.
  \end{proof}

  Now we return to the main line of proof. Let us first consider the case when
  \begin{equation}\label{eqn.AMinHyp}
  \min_{1\le i\le d}A_i~>~\max_{1\le j\le d}B_j.
  \end{equation}
  With a view towards applying Lemma \ref{lem.ElemNumThy} in order to understand the points of the sets $S_\mathcal{I}$, for $\mathcal{I}\subset\{1,\ldots ,d\}$, let
  \begin{equation}
  \mathcal{A}_\mathcal{I}=\left\{\sum_{i\in\mathcal{I}}A_iB_i+\sum_{i\notin\mathcal{I}}a_iB_i~:~0\le a_i< A_i,i\notin\mathcal{I}\right\}.
  \end{equation}
  Setting
  \begin{equation}
  D_\mathcal{I}=\sum_{i\in\mathcal{I}}A_iB_i+\sum_{i\notin\mathcal{I}}(A_i-1)B_i,
  \end{equation}
  we have by the lemma that
  \begin{equation}\label{eqn.S_phiInc}
  \left\{\sum_{i=1}^{d-1}B_iB_{i+1}\le m \le D_\emptyset-\sum_{i=1}^{d-1}B_iB_{i+1} \right\}\subset \mathcal{A}_\emptyset
  \end{equation}
  and, for any $\mathcal{I}$, that
  \begin{equation}\label{eqn.S_IInc}
  \mathcal{A}_\mathcal{I}\subset\left\{0\le m\le D_\mathcal{I}\right\}\subset\left\{0\le m\le D_\mathcal{\emptyset}+\sum_{i=1}^d B_i\right\}.
  \end{equation}
  Now, comparing the definitions of $\mathcal{A}_\mathcal{I}$ with the descriptions of the corresponding sets $S_\mathcal{I}$ from \eqref{eqn.S_IDesc}, we see that each set $S_\mathcal{I}$ consists of points of the form
  \begin{equation}\label{eqn.GenPtForm}
    m\beta+\sum_{i=1}^dr_i\alpha_i,
  \end{equation}
  with $m\in\mathcal{A}_\mathcal{I}$ and with each parameter $r_i$ taken either from the interval $[1,R_i]$ or from $[1,Q]$. From \eqref{eqn.S_phiInc} we see that $S_\emptyset$ contains all points of the set
  \begin{equation}
    S'=\left\{m\beta+\sum_{i=1}^dr_i\alpha_i~:~\sum_{i=1}^{d-1}B_iB_{i+1}\le m \le D_\emptyset-\sum_{i=1}^{d-1}B_iB_{i+1},~1\le r_i\le Q\right\}.
  \end{equation}
  Furthermore, by \eqref{eqn.S_IInc} we see that any other point of the form \eqref{eqn.GenPtForm}, which is included in one of the sets $S_\mathcal{I}$ but not in $S'$, must have
  \begin{equation}
    0\le m<\sum_{i=1}^{d-1}B_iB_{i+1}\qquad\text{or}\qquad D_\emptyset-\sum_{i=1}^{d-1}B_iB_{i+1}<m\le D_\emptyset+\sum_{i=1}^dB_i
  \end{equation}
  and
  \begin{equation}
    1\le r_i\le Q,
  \end{equation}
  for each $i$. The number of such points is bounded above by a constant which depends only on $B_1,\ldots ,B_d,$ and $Q$.

  At this point we have shown that the set $S(\alpha,\mathcal{D}_T)$ can be written as
  \begin{equation}
  S(\alpha,\mathcal{D}_T)=S'\cup S'',
  \end{equation}
  with $S'$ as above, and with the set $S''$ containing no more than $C=C(B_1,\ldots , B_d,Q)$ elements. Now using Chevallier's result \eqref{eqn.Chev} with $d$ replaced by $d+1,~N_1=\cdots=N_d=Q$, and
  \begin{equation}
    N_{d+1}=D_\emptyset-2\sum_{i=1}^{d-1}B_iB_{i+1}+1,
  \end{equation}
  we see that the number of distinct gaps between consecutive elements of $S'$ is at most
  \begin{equation}
  Q^d+3Q^{d-1}+1.
  \end{equation}
  Each element of $S''$ can divide at most one of these gaps, creating at most two new distinct gaps. Therefore we have proved that
  \begin{equation}
  G(\alpha,\mathcal{D}_T)\le Q^d+3Q^{d-1}+1+2C.
  \end{equation}
  This completes the proof in the case when \eqref{eqn.AMinHyp} holds. The remaining cases are no more difficult. If it happens that one or more of the quantities $A_i$ is chosen so that
  \begin{equation}
  A_i\le\max_{1\le j\le d}B_j,
  \end{equation}
  then the corresponding value of $M_i$ is also bounded by a constant which only depends on $B_1,\ldots ,B_d$ and $Q$. In this case we may ignore this index $i$ in our construction of the sets $S_\mathcal{I}$, until the end when we may apply the same argument as before. This therefore completes the proof.

\section{Proofs of Theorems \ref{thm:oneB}-\ref{thm:threeB}}\label{sec:Slater}
The proofs of our higher dimensional Slater theorems are simple adaptations of the machinery which we have developed. Using the notation from the Introduction, note that
\begin{align}
\tau(q,\scrD) & = \min\{ n>0 \mid q+ n\alpha +m \in\scrD, \; (m,n)\in\ZZ^{d+1} \} \\
& = \min\{ y > 0 \mid x + q \in\scrD, \; (x,y)\in\ZZ^{d+1} \tilde A_1 \} ,
\end{align}
with
\begin{equation}
\tilde A_1 =\begin{pmatrix} 1_d & \trans 0 \\ \alpha & 1 \end{pmatrix} ,
\end{equation}
and therefore
\begin{equation}
\tau(q,\scrD_B) = (\det B)^{-1} \min\{ y > 0 \mid x + q B^{-1} \in\scrD, \; (x,y)\in\ZZ^{d+1} \tilde A_B \}
\end{equation}
with
\begin{equation}
\tilde A_B =\begin{pmatrix} 1_d & \trans 0 \\ \alpha & 1 \end{pmatrix} \begin{pmatrix} B^{-1} & 0 \\ 0 & \det B \end{pmatrix} .
\end{equation}
This shows that
\begin{equation}
\tau(q,\scrD_B) = (\det B)^{-1} F(\tilde A_B, q B^{-1}).
\end{equation}
For Theorem \ref{thm:oneB} we choose $B=\diag(R_i,\ldots,R_i)^{-1}$. Taking transpose-inverses of the matrices defining the lattices in \eqref{friday}, and using the fact that $P=-P$, we see that if $\alpha\in P$ then, with $t_i=\log R_i$, the set
\begin{equation}\label{fridayB}
\bigg\{ \Gamma \begin{pmatrix} 1_d & \trans 0 \\ \alpha & 1 \end{pmatrix} \Phi^{-t_i} \;\bigg|\; i\in\NN^* \bigg\}
\end{equation}
is dense in $\GamG$. Theorem \ref{thm:oneB} then follows from Corollary \ref{cor:dense} and Proposition \ref{prop:three} as before. Theorems \ref{thm:twoB} and \ref{thm:threeB} follow from the remaining arguments in Section \ref{sec:three}. We note, however, that the proof of Theorem \ref{thm:threeB} is actually simpler than the proof of the corresponding Theorem \ref{thm:three}. This is because, when we get to the equation analogous to \eqref{eqn.GapsBdd1}, we deduce that
\begin{equation}
\scrG\bigg(  \begin{pmatrix} 1_d & 0 \\ -\alpha & 1 \end{pmatrix} \begin{pmatrix} T & \trans 0 \\ 0 & (\det T)^{-1} \end{pmatrix} D(\theta)^{-1} \bigg) \leq C_\theta
\end{equation}
for all $T_1,\ldots,T_d\geq 1$ and for all $\theta>\overline\rho(\scrC',\scrD\times (0,1])$. This implies that
\begin{equation}
  L(\alpha,\mathcal{D}_{T^{-1}})\le C_\theta,
\end{equation}
for all $T_1,\ldots ,T_d\ge \theta^{-1}$. Since it is clear that we may take $\theta\ge 1$, this is all that is needed to complete the proof of Theorem \ref{thm:threeB}.

\vspace{.15in}

{\footnotesize
\noindent
AH: Department of Mathematics, University of Houston,\\
Houston, TX, United States.\\
haynes@math.uh.edu\\

\noindent
JM: School of Mathematics, University of Bristol,\\
Bristol, United Kingdom.\\
j.marklof@bristol.ac.uk
}

\end{document}

%% file: gapspaper1.pdf_tex
\begingroup%
  \makeatletter%
  \providecommand\color[2][]{%
    \errmessage{(Inkscape) Color is used for the text in Inkscape, but the package 'color.sty' is not loaded}%
    \renewcommand\color[2][]{}%
  }%
  \providecommand\transparent[1]{%
    \errmessage{(Inkscape) Transparency is used (non-zero) for the text in Inkscape, but the package 'transparent.sty' is not loaded}%
    \renewcommand\transparent[1]{}%
  }%
  \providecommand\rotatebox[2]{#2}%
  \ifx\svgwidth\undefined%
    \setlength{\unitlength}{389.31301862bp}%
    \ifx\svgscale\undefined%
      \relax%
    \else%
      \setlength{\unitlength}{\unitlength * \real{\svgscale}}%
    \fi%
  \else%
    \setlength{\unitlength}{\svgwidth}%
  \fi%
  \global\let\svgwidth\undefined%
  \global\let\svgscale\undefined%
  \makeatother%
  \begin{picture}(1,0.50076168)%
    \put(0,0){\includegraphics[width=\unitlength,page=1]{gapspaper1.pdf}}%
    \put(0.09944475,0.24864315){\color[rgb]{0,0,0}\makebox(0,0)[lb]{\smash{}}}%
    \put(0.17355506,0.30847786){\color[rgb]{0,0,0}\makebox(0,0)[lb]{\smash{$R=t$}}}%
    \put(0.25914876,0.4406612){\color[rgb]{0,0,0}\makebox(0,0)[lb]{\smash{$S$}}}%
    \put(0.24241747,0.38471943){\color[rgb]{0,0,0}\makebox(0,0)[lb]{\smash{$\ell u$}}}%
    \put(0.10581079,0.37479748){\color[rgb]{0,0,0}\makebox(0,0)[lb]{\smash{$\lambda u$}}}%
    \put(0.66084878,0.15161426){\color[rgb]{0,0,0}\makebox(0,0)[lb]{\smash{$O$}}}%
    \put(0.79550379,0.33056542){\color[rgb]{0,0,0}\makebox(0,0)[lb]{\smash{$P$}}}%
    \put(0.72793767,0.2361182){\color[rgb]{0,0,0}\makebox(0,0)[lb]{\smash{$Q$}}}%
    \put(0.76644312,0.18235588){\color[rgb]{0,0,0}\makebox(0,0)[lb]{\smash{$|OQ|=\ell$}}}%
    \put(-0.00155522,0.48494857){\color[rgb]{0,0,0}\makebox(0,0)[lb]{\smash{$\mathcal{D}:$}}}%
    \put(0.37514591,0.30664571){\color[rgb]{0,0,0}\makebox(0,0)[lb]{\smash{$\Delta\mathcal{D}:$}}}%
  \end{picture}%
\endgroup%

%% file: gapspaper2.pdf_tex
\begingroup%
  \makeatletter%
  \providecommand\color[2][]{%
    \errmessage{(Inkscape) Color is used for the text in Inkscape, but the package 'color.sty' is not loaded}%
    \renewcommand\color[2][]{}%
  }%
  \providecommand\transparent[1]{%
    \errmessage{(Inkscape) Transparency is used (non-zero) for the text in Inkscape, but the package 'transparent.sty' is not loaded}%
    \renewcommand\transparent[1]{}%
  }%
  \providecommand\rotatebox[2]{#2}%
  \ifx\svgwidth\undefined%
    \setlength{\unitlength}{511.15562129bp}%
    \ifx\svgscale\undefined%
      \relax%
    \else%
      \setlength{\unitlength}{\unitlength * \real{\svgscale}}%
    \fi%
  \else%
    \setlength{\unitlength}{\svgwidth}%
  \fi%
  \global\let\svgwidth\undefined%
  \global\let\svgscale\undefined%
  \makeatother%
  \begin{picture}(1,0.53666655)%
    \put(0,0){\includegraphics[width=\unitlength,page=1]{gapspaper2.pdf}}%
    \put(-0.28399105,0.39105729){\color[rgb]{0,0,0}\makebox(0,0)[lb]{\smash{}}}%
    \put(0.43203345,0.29129435){\color[rgb]{0,0,0}\makebox(0,0)[lb]{\smash{$O$}}}%
    \put(0.68965293,0.51412058){\color[rgb]{0,0,0}\makebox(0,0)[lb]{\smash{$P$}}}%
    \put(0.04726165,0.50432659){\color[rgb]{0,0,0}\makebox(0,0)[lb]{\smash{$\Delta\mathcal{D}:$}}}%
    \put(0,0){\includegraphics[width=\unitlength,page=2]{gapspaper2.pdf}}%
    \put(0.24678775,0.0051644){\color[rgb]{0,0,0}\makebox(0,0)[lb]{\smash{$-P$}}}%
    \put(0,0){\includegraphics[width=\unitlength,page=3]{gapspaper2.pdf}}%
    \put(0.66162638,0.45589486){\color[rgb]{0,0,0}\makebox(0,0)[lb]{\smash{$a_1=1, m=1$}}}%
    \put(0.63971524,0.40393415){\color[rgb]{0,0,0}\makebox(0,0)[lb]{\smash{$a_1=1, m=2$}}}%
    \put(0.60528346,0.35197353){\color[rgb]{0,0,0}\makebox(0,0)[lb]{\smash{$a_1=1, m=3$}}}%
    \put(0,0){\includegraphics[width=\unitlength,page=4]{gapspaper2.pdf}}%
    \put(0.48946058,0.23976713){\color[rgb]{0,0,0}\makebox(0,0)[lb]{\smash{$a_1=0, m=1$}}}%
    \put(0.46754945,0.18780642){\color[rgb]{0,0,0}\makebox(0,0)[lb]{\smash{$a_1=0, m=2$}}}%
    \put(0.43311766,0.13584581){\color[rgb]{0,0,0}\makebox(0,0)[lb]{\smash{$a_1=0, m=3$}}}%
  \end{picture}%
\endgroup%

%% file: gapspaper3v2.pdf_tex
\begingroup%
  \makeatletter%
  \providecommand\color[2][]{%
    \errmessage{(Inkscape) Color is used for the text in Inkscape, but the package 'color.sty' is not loaded}%
    \renewcommand\color[2][]{}%
  }%
  \providecommand\transparent[1]{%
    \errmessage{(Inkscape) Transparency is used (non-zero) for the text in Inkscape, but the package 'transparent.sty' is not loaded}%
    \renewcommand\transparent[1]{}%
  }%
  \providecommand\rotatebox[2]{#2}%
  \ifx\svgwidth\undefined%
    \setlength{\unitlength}{1170.35583281bp}%
    \ifx\svgscale\undefined%
      \relax%
    \else%
      \setlength{\unitlength}{\unitlength * \real{\svgscale}}%
    \fi%
  \else%
    \setlength{\unitlength}{\svgwidth}%
  \fi%
  \global\let\svgwidth\undefined%
  \global\let\svgscale\undefined%
  \makeatother%
  \begin{picture}(1,0.65888663)%
    \put(0,0){\includegraphics[width=\unitlength,page=1]{gapspaper3v2.pdf}}%
    \put(-0.12991957,0.56674414){\color[rgb]{0,0,0}\makebox(0,0)[lb]{\smash{}}}%
    \put(0.29305055,0.49466979){\color[rgb]{0,0,0}\makebox(0,0)[lb]{\smash{$\mathcal{D}-t_1'$}}}%
    \put(0,0){\includegraphics[width=\unitlength,page=2]{gapspaper3v2.pdf}}%
    \put(-0.21037121,0.74581201){\color[rgb]{0,0,0}\makebox(0,0)[lt]{\begin{minipage}{0.2652137\unitlength}\raggedright \end{minipage}}}%
    \put(0.02381316,0.65362647){\color[rgb]{0,0,0}\makebox(0,0)[lb]{\smash{$m=1:$}}}%
    \put(0,0){\includegraphics[width=\unitlength,page=3]{gapspaper3v2.pdf}}%
    \put(0.85082962,0.48920136){\color[rgb]{0,0,0}\makebox(0,0)[lb]{\smash{$\mathcal{D}-t_1$}}}%
    \put(0,0){\includegraphics[width=\unitlength,page=4]{gapspaper3v2.pdf}}%
    \put(-0.62174237,0.28424794){\color[rgb]{0,0,0}\makebox(0,0)[lb]{\smash{}}}%
    \put(0.25786104,0.12560017){\color[rgb]{0,0,0}\makebox(0,0)[lb]{\smash{$\mathcal{D}-t_2'$}}}%
    \put(0,0){\includegraphics[width=\unitlength,page=5]{gapspaper3v2.pdf}}%
    \put(-0.70219402,0.46331584){\color[rgb]{0,0,0}\makebox(0,0)[lt]{\begin{minipage}{0.2652137\unitlength}\raggedright \end{minipage}}}%
    \put(0.02381315,0.29271053){\color[rgb]{0,0,0}\makebox(0,0)[lb]{\smash{$m=2:$}}}%
    \put(0,0){\includegraphics[width=\unitlength,page=6]{gapspaper3v2.pdf}}%
    \put(0.81866083,0.12038942){\color[rgb]{0,0,0}\makebox(0,0)[lb]{\smash{$\mathcal{D}-t_2$}}}%
    \put(0,0){\includegraphics[width=\unitlength,page=7]{gapspaper3v2.pdf}}%
    \put(0.68303558,0.36365818){\color[rgb]{0,0,0}\makebox(0,0)[lb]{\smash{$F(M_\epsilon,t_1)=\epsilon$}}}%
    \put(0.68303558,0.00137512){\color[rgb]{0,0,0}\makebox(0,0)[lb]{\smash{$F(M_\epsilon,t_2)=2\epsilon$}}}%
  \end{picture}%
\endgroup%